\RequirePackage{fix-cm}
\documentclass[smallextended,envcountsect,envcountsame]{svjour3}

\usepackage{amsmath,amssymb,mathtools}
\usepackage{microtype}
\usepackage[hidelinks]{hyperref}

\smartqed
\renewenvironment{proof}[1][Proof]
  {\par\addvspace{\medskipamount}\noindent\textit{#1.}\enspace\normalfont}
  {\qed\par\addvspace{\medskipamount}}
\renewcommand{\makeheadbox}{}
\renewcommand{\subclassname}{\textbf{Mathematics Subject Classification (2020)}\enspace}
\journalname{Journal of Optimization Theory and Applications}
\date{}

\newcommand{\R}{\mathrm{I\!R}}
\newcommand{\trans}{\mathsf{T}}
\newcommand{\norm}[1]{\lVert#1\rVert_2}

\title{Counterexamples to $R$-Superlinear Convergence of Cyclic Steepest Descent}
\titlerunning{Counterexamples for Cyclic Steepest Descent}

\author{Yu Li \and Qihang Wang}
\institute{Yu Li \at
Key Laboratory of Quantum Theory and Applications of MoE, Lanzhou Center for
Theoretical Physics, and Key Laboratory of Theoretical Physics of Gansu Province,
Lanzhou University, Lanzhou, 730000, Gansu, China\\
\email{bosonization@163.com}
\and Qihang Wang, Corresponding author \at
School of Mathematical Sciences, Peking University, Beijing 100871, China\\
\email{2501110049@stu.pku.edu.cn}}

\hypersetup{
  pdftitle={Counterexamples to R-Superlinear Convergence of Cyclic Steepest Descent},
  pdfauthor={Yu Li and Qihang Wang},
  pdfsubject={Exact convergence analysis of cyclic steepest descent},
  pdfkeywords={cyclic steepest descent, convex quadratic optimization, R-superlinear convergence, counterexample, periodic orbit}
}

\begin{document}

\maketitle
\thispagestyle{plain}

\begin{abstract}
The cyclic steepest descent method computes the exact steepest descent
stepsize at the start of each cycle and reuses it for a fixed number $m$ of
iterations.  For strictly convex quadratics in $n$ variables, it has been
suggested that the method converges $R$-superlinearly whenever
$m\ge(n+1)/2$.  We show that this cannot hold for every initial point, already
at the threshold cycle length.  For $n=m=2$ we exhibit an initial point from
which the method follows the classical zigzag trajectory of steepest descent:
the stepsize never changes and the error is halved at every iteration.  The
same mechanism yields linearly convergent trajectories for every Hessian with
at least two distinct eigenvalues and every cycle length.  For $n=3$ and
$m=2$ we construct a second, genuinely three-dimensional counterexample: all
components of the gradient remain nonzero, consecutive cycles use two
different stepsizes, and every four iterations reduce the gradient exactly by
the factor $1/49$.  Both trajectories are given in closed form.  Consequently,
any $R$-superlinear convergence result for this method at the threshold must
exclude a set of initial points, as recent almost-everywhere results do.
\end{abstract}


\section{Introduction}

Cyclic steepest descent (CSD) reuses an exact steepest-descent stepsize for a
prescribed number of consecutive updates.  It is a special delayed gradient
method in the sense of Friedlander, Mart\'inez, Molina, and
Raydan~\cite{friedlander1999}; Dai and Fletcher studied its asymptotic behavior
on positive-definite quadratics~\cite{daifletcher2005}.  With reference to that
analysis, Dai's ICM 2022 survey states that CSD is likely to converge
$R$-superlinearly when the cycle length satisfies
$m\ge\lceil(n+1)/2\rceil$~\cite{daiicm2022}; the survey writes the condition
as $m\ge(n+1)/2$, which for integer $m$ is the ceiling condition used
here and is equivalent to $n<2m$, the superlinear side of the threshold in
the simplified model of Dai and Fletcher.  The ICM Conjectures directory
records the corresponding universal assertion as problem
2022-opt-18~\cite{icmconjectures}.
Our counterexamples address this universal formulation, not an
almost-everywhere interpretation of the survey's conjectural statement.

Theorem~\ref{thm:p5640} gives a counterexample at the smallest nontrivial
dimension and cycle length (for $n=1$ the method terminates after one
update), with error norm ratio $1/2$ at every actual update.  The statement describes the complete orbit, not a finite numerical
trace.  Proposition~\ref{prop:invariant-orbit} extends this invariant family
to arbitrary pairs of eigendirections and cycle lengths.

Theorem~\ref{thm:three-spectrum} gives a second obstruction at the
three-dimensional threshold: $A=\operatorname{diag}(1,2,3)$ and $m=2$.
All three spectral components remain nonzero.  The cycle-boundary weights
alternate between two interior points of the simplex, and every four
updates multiply the gradient and error by $1/49$.  The proof identifies
the two Rayleigh quotients and verifies the two weight transitions exactly.
Thus both counterexamples describe entire trajectories, but only the first
uses a constant stepsize.

The underlying spectral mechanism is classical for ordinary exact-line-search steepest
descent.  Akaike's probability-map analysis contains the equal-weight fixed
distribution~\cite{akaike1959}, and Yu recently stated the corresponding
worst-case linear orbit and its higher-dimensional embedding explicitly
\cite[Proposition~3.1 of the arXiv version]{yu2025}.  The decisive observation for CSD is that the
exact stepsize is constant on this orbit.  Consequently, cyclic reuse for any
number of updates preserves the orbit exactly, turning the classical
equal-weight trajectory into a threshold counterexample to the universal CSD
assertion.

Periodic behavior under cyclic reuse has a direct precedent.
Dai, Hager, Schittkowski, and Zhang~\cite[Section~3.3]{cbb2006}
identify the alternate-step method as CSD with $m=2$ and construct a
five-dimensional cyclic example.  Theorem~\ref{thm:three-spectrum}
provides a closed-form full-support example at the three-dimensional
threshold; periodic CSD behavior itself is already documented.
The alternate minimization method of Dai and Yuan~\cite{daiyuan2003}
uses different steps and has a two-dimensional $Q$-superlinear result.
Shen, Li, and Dai~\cite[Theorem~1]{shen2025v1} state $R$-superlinear
convergence of the cycle-boundary gradient subsequence for two-dimensional
CSD; their argument parameterizes the initial gradient by the logarithm of
its component ratio and uses that this logarithm is nonzero.  The balanced
initializations of Theorem~\ref{thm:p5640} and
Proposition~\ref{prop:invariant-orbit} are exactly the fixed point of their
recurrence at which this logarithm vanishes, where the cycle-boundary subsequence is
geometric; this boundary case is treated here.  Their revised preprint
studies the triangle steepest-descent method instead~\cite{shen2025v2}.

The companion paper~\cite[Theorem~3.1]{liwangpositive2026} classifies CSD
for Hessians with exactly two distinct eigenvalues and arbitrary
eigenspace multiplicities.  For $m\ge2$, the equal-norm spectral-block
cone is precisely the nonterminating exceptional set; every other mixed
initialization is $R$-superlinear.  In the fixed two-dimensional instance
here this cone is the union $|g_{0,1}|=|g_{0,2}|>0$ of two lines.
The two papers therefore separate the construction of exact exceptions from
the classification of all initializations for two-spectrum Hessians.
A recent preprint of Gu~\cite[Theorem~2.1]{gu2026} states an
almost-everywhere doubly exponential convergence theorem for general
spectra when $2m$ exceeds the number of distinct eigenvalues.
Both instances here satisfy that condition, so by that theorem the orbits
constructed below must lie in its Lebesgue-null exceptional set; by
Theorem~\ref{thm:three-spectrum} that set can contain full-support,
nonresonant orbits, so it cannot be dropped.  Exact exceptional orbits and
almost-everywhere convergence are compatible.  Neither positive result is
needed in our proofs.

Section~\ref{sec:csd} fixes the method, the notation, and the rate
terminology.  Section~\ref{sec:examples} proves the two counterexamples
and the invariant family.  Section~\ref{sec:conclusions} summarizes the
consequences.

\section{Cyclic Steepest Descent}
\label{sec:csd}

On $\R^n$, let
\begin{equation}\label{eq:quadratic}
  q(x)=\frac12 x^{\trans}Ax-c^{\trans}x,
  \qquad A=A^{\trans}\succ0,
\end{equation}
where $c\in\R^n$, and write $g_k=Ax_k-c$.
All norms are Euclidean.
The method terminates whenever $g_k=0$.
Otherwise, the exact steepest-descent stepsize is
\begin{equation}\label{eq:sd-step}
  \alpha_k^{\mathrm{SD}}
  :=\frac{g_k^{\trans}g_k}{g_k^{\trans}Ag_k}.
\end{equation}

Let $m\ge1$ be an integer.  At the boundary $m\ell$ of cycle $\ell$, CSD
computes $\alpha_{m\ell}^{\mathrm{SD}}$ from~\eqref{eq:sd-step} and performs
the $m$ updates
\begin{equation}\label{eq:csd}
  x_{m\ell+i+1}
  =x_{m\ell+i}-\alpha_{m\ell}^{\mathrm{SD}}g_{m\ell+i},
  \qquad i=0,\ldots,m-1.
\end{equation}
Thus a newly computed exact stepsize is used for exactly $m$ consecutive
updates.  Equivalent index shifts do not change the method.  The unique
minimizer of~\eqref{eq:quadratic} is $x_*=A^{-1}c$.  Set
\begin{equation}\label{eq:error-gradient}
  e_k=x_k-x_*,\qquad g_k=Ae_k.
\end{equation}
For any update with stepsize $\alpha_k$,
\begin{equation}\label{eq:gradient-update}
  g_{k+1}=(I-\alpha_k A)g_k.
\end{equation}

A null sequence $\{e_k\}$ converges $R$-superlinearly if
$\norm{e_k}^{1/k}\to0$, that is, if its root-convergence factor
$\limsup_k\norm{e_k}^{1/k}$ is zero~\cite[Chapter~9]{ortega2000};
equivalently, $\norm{e_k}$ is bounded by a null sequence whose successive
ratios tend to zero, which is the form used in~\cite{shen2025v1}.  It
converges $R$-linearly if this root factor lies in $(0,1)$.
A positive geometric lower bound on the errors, even along a fixed-stride
subsequence, forces a positive root factor and hence excludes
$R$-superlinear convergence.  This elementary observation suffices for the
exact orbits below.

\section{Exact Counterexamples}
\label{sec:examples}

\subsection{The Balanced Zig-Zag Obstruction}

\begin{theorem}[Exact threshold counterexample]
  \label{thm:p5640}
  Let
  \begin{equation}\label{eq:p5640-data}
    n=2,\qquad m=2,\qquad
    A=\begin{pmatrix}1&0\\0&3\end{pmatrix},\qquad
    c=\begin{pmatrix}0\\0\end{pmatrix},\qquad
    x_0=\begin{pmatrix}1\\1/3\end{pmatrix}.
  \end{equation}
  Every stepsize recomputed by CSD is $1/2$, and the complete orbit is
  \begin{align}
    g_k&=2^{-k}\begin{pmatrix}1\\(-1)^k\end{pmatrix},
      \label{eq:p5640-gradient}\\
    x_k=e_k&=2^{-k}\begin{pmatrix}1\\(-1)^k/3\end{pmatrix}.
      \label{eq:p5640-error}
  \end{align}
  It is nonterminating and satisfies
  \begin{equation}\label{eq:p5640-ratio}
    \frac{\norm{e_{k+1}}}{\norm{e_k}}=\frac12
    \qquad(k\ge0).
  \end{equation}
  Consequently, the assertion of universal $R$-superlinear convergence for
  $m\ge\lceil(n+1)/2\rceil$ does not hold.
\end{theorem}

\begin{proof}[Proof of Theorem~\ref{thm:p5640}]
Here $x_*=0$ and $g_0=(1,1)^{\trans}$.  Whenever a cycle starts with two
gradient components of equal magnitude, its recomputed stepsize is
\[
  \frac{g^{\trans}g}{g^{\trans}Ag}=\frac12.
\]
For this stepsize,
\[
  I-\frac12 A=\begin{pmatrix}1/2&0\\0&-1/2\end{pmatrix}.
\]
The two component magnitudes therefore remain equal after each of the two
updates in a cycle.  Induction over the cycles gives
\eqref{eq:p5640-gradient}; applying $A^{-1}$ gives
\eqref{eq:p5640-error}, and~\eqref{eq:p5640-ratio} follows.  In particular
$\norm{e_k}=(\sqrt{10}/3)\,2^{-k}$, so the error sequence has
exactly $R$-linear convergence with factor $1/2$ and hence, by the
observation of Section~\ref{sec:csd}, is not $R$-superlinearly convergent.  Finally, $m=2=\lceil(2+1)/2\rceil$ is the
smallest admissible integer cycle length, so the universal assertion fails at
its threshold.  The proof is complete.
\end{proof}

The same calculation extends to any pair of eigendirections and any
cycle length.

\begin{proposition}[Two-eigendirection invariant family]
  \label{prop:invariant-orbit}
  Let $A\succ0$ have orthonormal eigenvectors $u,v$ with distinct eigenvalues
  $0<a<b$, respectively.  Fix any cycle length $m\ge1$, any
  $t\in\R\setminus\{0\}$, and any $\sigma\in\{-1,1\}$.  Initialize CSD so
  that
  \begin{equation}\label{eq:initial-gradient-general}
    g_0=t(u+\sigma v).
  \end{equation}
  Define
  \begin{equation}\label{eq:rho}
    \rho:=\frac{b-a}{b+a}\in(0,1).
  \end{equation}
  Then every stepsize recomputed at a cycle boundary is
  \begin{equation}\label{eq:constant-step}
    \alpha_{m\ell}^{\mathrm{SD}}=\frac{2}{a+b},
    \qquad \ell=0,1,\ldots,
  \end{equation}
  and the full orbit is
  \begin{align}
    g_k&=t\rho^k\bigl(u+(-1)^k\sigma v\bigr),
      \label{eq:gradient-orbit}\\
    e_k&=t\rho^k\bigl(a^{-1}u
      +(-1)^k\sigma b^{-1}v\bigr).
      \label{eq:error-orbit}
  \end{align}
  In particular,
  \begin{equation}\label{eq:constant-ratios}
    \frac{\norm{g_{k+1}}}{\norm{g_k}}
    =\frac{\norm{e_{k+1}}}{\norm{e_k}}
    =\rho
    \qquad(k\ge0).
  \end{equation}
  None of the full gradient sequence, the full error sequence, and their
  cycle-boundary subsequences converges $R$-superlinearly.
\end{proposition}

\begin{proof}[Proof of Proposition~\ref{prop:invariant-orbit}]
For a vector of the form $s(u+\tau v)$ with $s\ne0$ and
$\tau\in\{-1,1\}$, orthonormality and the eigenvalue relations give
\[
  \frac{g^{\trans}g}{g^{\trans}Ag}
  =\frac{2s^2}{(a+b)s^2}
  =\frac{2}{a+b}.
\]
For this stepsize,
\[
  (I-\alpha A)u=\rho u,
  \qquad
  (I-\alpha A)v=-\rho v.
\]
Equation~\eqref{eq:gradient-update} therefore changes the sign of the
$v$-component and multiplies both component magnitudes by $\rho$.  After any
number of updates, and hence at every cycle boundary, the two components
still have equal magnitude.  Induction proves~\eqref{eq:constant-step} and
\eqref{eq:gradient-orbit}.  Applying $A^{-1}$ gives
\eqref{eq:error-orbit}.

The norms of the two parenthesized vectors in
\eqref{eq:gradient-orbit} and~\eqref{eq:error-orbit} are independent of $k$,
so~\eqref{eq:constant-ratios} follows.  The full sequences are geometric with
factor $\rho>0$, while their cycle-boundary subsequences are geometric with
factor $\rho^m>0$ in the cycle index; by the observation of
Section~\ref{sec:csd}, none of them converges $R$-superlinearly.
The proof is complete.
\end{proof}

\subsection{A Full-Support Non-Zig-Zag Obstruction}

\begin{theorem}[Full-support three-spectral period-two counterexample]
  \label{thm:three-spectrum}
  Let
  \begin{equation}\label{eq:three-spectrum-data}
    n=3,\qquad m=2,\qquad A=\operatorname{diag}(1,2,3),\qquad c=0,
  \end{equation}
  and set
  \begin{equation}\label{eq:three-spectrum-initial}
    g_0=\begin{pmatrix}(2+\sqrt2)/4\\[2pt]1/2\\[2pt](2-\sqrt2)/4\end{pmatrix},
    \qquad
    x_0=A^{-1}g_0
      =\begin{pmatrix}(2+\sqrt2)/4\\[2pt]1/4\\[2pt](2-\sqrt2)/12\end{pmatrix}.
  \end{equation}
  Define
  \[
    \mu_- :=2-\frac{\sqrt2}{2},\qquad
    \mu_+ :=2+\frac{\sqrt2}{2},\qquad
    \alpha_\pm:=\frac1{\mu_\pm}.
  \]
  The CSD orbit is nonterminating, has full support (every component of
  $g_k$ is nonzero for every $k$), and is nonresonant: no cycle-boundary
  stepsize annihilates a spectral component, i.e., $\mu_\pm\ne\lambda_i$ for
  every $i$.  The orbit has projective period two at cycle boundaries: writing
  $w_{k,i}=g_{k,i}^2/\sum_{j=1}^3g_{k,j}^2$, the normalized squared component
  weights $w_{2\ell}$ alternate between two distinct interior points of the simplex.
  The stepsize pattern is
  $\alpha_-,\alpha_-,\alpha_+,\alpha_+,\ldots$, and
  \begin{equation}\label{eq:three-spectrum-contraction}
    g_{k+4}=\frac1{49}g_k,\qquad e_{k+4}=\frac1{49}e_k
    \qquad(k\ge0).
  \end{equation}
  Consequently the orbit is $R$-linearly convergent with root
  factor $49^{-1/4}=1/\sqrt7$, but not $R$-superlinearly convergent, and it
  is not a two-dimensional balanced-zigzag orbit.
\end{theorem}

\begin{proof}[Proof of Theorem~\ref{thm:three-spectrum}]
  For a nonzero gradient $g$, write
  $w_i=g_i^2/\sum_{j=1}^3g_j^2$ and
  $\mu=\sum_{i=1}^3\lambda_iw_i$, where
  $(\lambda_1,\lambda_2,\lambda_3)=(1,2,3)$.
  At a cycle boundary the exact stepsize is
  $\alpha=1/\mu$; after two updates the $i$th gradient component is multiplied
  by $(1-\lambda_i/\mu)^2$.  Therefore the normalized squared component
  weights obey the map
  \begin{equation}\label{eq:three-spectrum-weight-map}
    T_i(w)=\frac{w_i|\mu-\lambda_i|^4}
      {\sum_jw_j|\mu-\lambda_j|^4}.
  \end{equation}
  Let $w^-=(g_{0,1}^2,g_{0,2}^2,g_{0,3}^2)$, whose entries sum to one, so
  that $w^-$ is the initial weight vector; let $w^+$ be obtained by
  interchanging $\sqrt2$ with $-\sqrt2$ in the three entries.  Explicitly,
  \[
    w^-=\left(\frac{3+2\sqrt2}{8},\frac14,\frac{3-2\sqrt2}{8}\right),
    \qquad
    w^+=\left(\frac{3-2\sqrt2}{8},\frac14,\frac{3+2\sqrt2}{8}\right).
  \]
  Both vectors are positive and sum to one, so they are interior points of the
  simplex.  Their weighted means are $\mu_-$ and $\mu_+$, respectively.
  Since
  \[
    (1-\mu_-)(1-\mu_+)=\frac12,\quad
    (2-\mu_-)(2-\mu_+)=-\frac12,\quad
    (3-\mu_-)(3-\mu_+)=\frac12,
  \]
  direct substitution gives
  \[
    w_i^-|\mu_--\lambda_i|^4=\frac14w_i^+,
    \qquad
    w_i^+|\mu_+-\lambda_i|^4=\frac14w_i^-.
  \]
  Hence $T(w^-)=w^+$ and $T(w^+)=w^-$, so
  $w_{4\ell}=w^-$ and $w_{4\ell+2}=w^+$ for every $\ell\ge0$.
  This proves the cycle-boundary projective period two
  and the displayed stepsize pattern.  Since $\mu_\pm\notin\{1,2,3\}$, each
  update multiplies every gradient component by a nonzero factor
  $1-\lambda_i\alpha_\pm$; hence all three components are nonzero at every
  iteration, so the orbit is nonterminating, has full support, and is
  nonresonant.  Moreover
  $\mu_-\mu_+=7/2$, so over the four updates the $i$th gradient component is
  multiplied by
  \[
    \left[\frac{(\mu_--\lambda_i)(\mu_+-\lambda_i)}
      {\mu_-\mu_+}\right]^2=\frac1{49}
  \]
  for each $i$.  Any four consecutive updates contain two factors of each
  type, and these diagonal update matrices commute.  Thus the identity holds
  for every starting index $k$, proving~\eqref{eq:three-spectrum-contraction}; applying
  $A^{-1}$ gives the error relation.  The subsequence $\{\norm{e_{4k}}\}$ is
  positive geometric with ratio $1/49$, so by the observation of
  Section~\ref{sec:csd} the orbit is not $R$-superlinearly convergent, while
  \eqref{eq:three-spectrum-contraction} gives the stated $R$-linear root
  factor $49^{-1/4}$.  The proof is complete.
\end{proof}

\begin{remark}[Higher-dimensional embedding]
For every $n\ge2$ and $m\ge1$, augment
$\operatorname{diag}(a,b)$ by arbitrary positive diagonal entries and
initialize the added gradient components at zero.  They remain zero, so the
active plane follows Proposition~\ref{prop:invariant-orbit}, and the
obstruction occurs in every dimension and at every cycle length as a
lower-dimensional support phenomenon.  Theorem~\ref{thm:three-spectrum} is
different: all three spectral components are nonzero, so its obstruction is
full-support.  Either theorem alone disproves the universal statement.
\end{remark}

\section{Conclusions}
\label{sec:conclusions}

Cyclic reuse preserves the classical balanced orbit and also admits the
full-support three-dimensional periodic orbit constructed above.
Together these examples show why a universal convergence guarantee must
account for exact exceptional initializations, even at the proposed cycle
threshold.  The three-dimensional construction also shows that full spectral
support does not exclude periodic projective dynamics.

Sampling from a distribution absolutely continuous with respect to Lebesgue
measure misses a fixed Lebesgue-null set with probability one, whereas
prescribed or discretely sampled initializations may lie on it.
Exceptional-orbit checks therefore matter for guarantees covering a specified
initialization, alongside almost-everywhere convergence results.

\begin{acknowledgements}

We thank Yu-Hong Dai for helpful comments on the exceptional-orbit calculation,
its interpretation as a classical zigzagging trajectory, and the distinction
between universal and almost-everywhere convergence.
We acknowledge Gewu Intelligence Lab for providing the collaborative research
environment in which this project was developed.
\end{acknowledgements}

\section*{Declarations}

\noindent\textbf{Funding.} No funding was received for this work.

\noindent\textbf{Competing interests.} The authors declare no competing interests.

\noindent\textbf{Data availability.} No datasets were generated or analysed.  The results are established
analytically in the article.

\section*{Disclosure of automated assistance}

The main results of this paper were obtained through a generative-AI
workflow using OpenAI GPT-5.6 Sol, Anthropic Claude Fable 5, OpenAI GPT-6
Astra, OpenAI GPT-6.1 Sol, Anthropic Claude Fable 5.1, and Anthropic Claude
Opus 5.5.
The workflow contributed to identifying and checking the exact
counterexamples, reviewing the cited literature, and preparing and revising
the LaTeX manuscript.

\end{document}